\documentclass[12pt,oneside,reqno]{article}
\usepackage[width=15cm,height=24cm]{geometry}

\usepackage{amsmath,amsthm,amssymb,mathtools}
\usepackage[authoryear,longnamesfirst]{natbib}
\usepackage{bbm}
\usepackage{mathrsfs}
\usepackage{textcase}
\usepackage{paralist}
\usepackage[font=small,labelfont=bf]{caption}
\usepackage{hyperref}
\usepackage{subfigure}
\usepackage{type1cm}

\usepackage{tikz,pgf}

\setdefaultleftmargin{0.5em}{}{}{}{}{}

\numberwithin{equation}{section}
\allowdisplaybreaks[4]

\theoremstyle{plain}
\newtheorem{theorem}{Theorem}[section]

\newtheorem{proposition}[theorem]{Proposition}
\newtheorem{lemma}[theorem]{Lemma}
\newtheorem{corollary}[theorem]{Corollary}

\theoremstyle{definition}
\newtheorem{definition}[theorem]{Definition}

\newtheorem{remark}[theorem]{Remark}

\makeatletter

\renewcommand{\cite}{\citet*}

\def\^#1{\ifmmode {\mathaccent"705E #1} \else {\accent94 #1} \fi}
\def\~#1{\ifmmode {\mathaccent"707E #1} \else {\accent"7E #1} \fi}

\def\*#1{#1^\ast}
\edef\-#1{\noexpand\ifmmode {\noexpand\bar{#1}} \noexpand\else
\-#1\noexpand\fi}\def\>#1{\vec{#1}}
\def\.#1{\dot{#1}}

\def\atop{\@@atop}
\def\%#1{\mathcal{#1}}

\renewcommand{\leq}{\leqslant}
\renewcommand{\geq}{\geqslant}
\renewcommand{\phi}{\varphi}

\newcommand{\eq}{\eqref}

\newcommand{\dw}{\mathop{d_{\mathrm{W}}}}
\newcommand{\dk}{\mathop{d_{\mathrm{K}}}}

\newcommand{\IE}{\mathbbm{E}}

\newcommand{\law}{\mathscr{L}}

\newcommand{\IR}{\mathbbm{R}}

\newcount\minute
\newcount\hour
\newcount\hourMins
\def\now{%
\minute=\time%
\hour=\time \divide \hour by 60%
\hourMins=\hour \multiply\hourMins by 60%
\advance\minute by -\hourMins%
\zeroPadTwo{\the\hour}:\zeroPadTwo{\the\minute}%
}
\def\zeroPadTwo#1{\ifnum #1<10 0\fi#1}

\renewcommand\section{\@startsection {section}{1}{\z@}%
{-3.5ex \@plus -1ex \@minus -.2ex}%
{1.3ex \@plus.2ex}%
{\center\small\sc\MakeTextUppercase}}

\def\subsection#1{\@startsection {subsection}{2}{0pt}%
{-3.5ex \@plus -1ex \@minus -.2ex}%
{1ex \@plus.2ex}%
{\bf\mathversion{bold}}{#1}}

\def\subsubsection#1{\@startsection{subsubsection}{3}{0pt}%
{\medskipamount}%
{-10pt}%
{\normalsize\itshape}{\kern-2.2ex. #1.}}

\def\blfootnote{\xdef\@thefnmark{}\@footnotetext}

\makeatother

\def\He{\text{He}}

\begin{document}

\title{\sc\bf\large\MakeUppercase{
 Stein's Method, Many Interacting Worlds and Quantum Mechanics
}
}
\author{\sc Ian W. McKeague, Erol Pek\"oz,
and Yvik Swan}
\date{\it Columbia University, Boston University, and  Universit\'{e} de Li\`{e}ge
}

\maketitle

\begin{abstract}
\cite{Hal2014} recently proposed that quantum theory can be understood as the continuum limit of a deterministic theory in which there is a large, but finite, number of classical ``worlds." A resulting Gaussian limit theorem for particle positions in the ground state, agreeing with quantum theory, was conjectured in \cite{Hal2014} and proven by \cite{McKeague2015} using Stein's method. In this article we propose new connections between Stein's method and Many Interacting Worlds (MIW) theory. In particular, we show that quantum position probability densities for higher energy levels beyond the ground state may arise as distributional fixed points in a new generalization of Stein's method.  These are then used to obtain a rate of distributional convergence for conjectured particle positions in the first energy level above the ground state to the (two-sided) Maxwell distribution; new techniques must be developed for this setting where the usual ``density approach'' Stein solution (see \cite{Chatterjee2011a}) has a singularity.
\end{abstract}


\section{Introduction}
 \cite{Hal2014}  proposed a many interacting worlds (MIW) theory for interpreting quantum mechanics in terms of a large but finite number of classical ``worlds."  In the case of the MIW harmonic oscillator, an energy minimization argument was used to derive a recursion giving the location of the oscillating particle as viewed in each of the worlds.  Hall et al.\ conjectured that the empirical distribution of these locations converges to Gaussian as the total number of worlds $N$ increases.  \cite{McKeague2015} recently proved such a result and provided a rate of convergence.
More specifically, McKeague and Levin showed that if $x_1, \ldots x_N$ is a decreasing, zero-mean sequence of real numbers satisfying the recursion relation \begin{equation}\label{e1}x_{n+1}=x_n-\frac{1}{x_1+\cdots +x_n},\end{equation}
then the empirical distribution of the $x_n$ tends to standard Gaussian  when $N\to \infty$.  Here $x_n$ represents the location of the oscillating particle in the $n$th world, and the Gaussian limit distribution agrees with quantum theory for a particle in the lowest energy (ground) state.

The hypothesized correspondence with quantum theory suggests that stable configurations should also exist at higher energies in the MIW theory. Moreover, the empirical distributions of these configurations should converge to distributions with densities of the form
\begin{equation}
  \label{eq:19}
p_k(x) = \frac{({\rm He}_k(x))^2}{k!} \varphi(x), \ \ x\in \mathbb{R},
\end{equation}
where $\varphi(x)$  is the
standard normal density, $$ \He_k(x)=(-1)^k e^{x^2/2}\frac{d^k}{dx^k} e^{-x^2/2}$$ is the (probabilist's) $k$th  Hermite polynomial, and $k$ is a non-negative integer. The ground state discussed above corresponds to $k=0$ and has the standard Gaussian limit.  However, the question of how to characterize  higher energy MIW states corresponding to $k\ge 1$ is still unresolved as far as we know.

 The energy minimization approach of \cite{Hal2014} starts with an analysis of the Hamiltonian for the MIW harmonic oscillator:
 $$H_0({\bf x},{\bf p}) =  E({\bf p}) +V({\bf x}) + U_0({\bf x}), $$
where the locations of  particles (having unit mass) in the $N$ worlds are specified by  ${\bf x} =(x_1,\ldots , x_N)$  with $x_1>x_2>\ldots > x_N$, and their momenta by   ${\bf p} =(p_1,\ldots,p_N)$.
Here  $E({\bf p}) = \sum_{n=1}^N p_n^2/2$ is the  kinetic energy,
$V({\bf x}) = \sum_{n=1}^N x_n^2$
is the  potential energy (for the parabolic trap), and
$$U_0({\bf x}) = \sum_{n=1}^N\left({1\over x_{n+1}-x_n} -{1\over x_{n}-x_{n-1}}\right)^2$$
is  called the ``interworld" potential, where
$x_0=\infty$ and $x_{N+1}=-\infty$.
In the ground state, there is no movement because all the momenta $p_n$ have to vanish for the total energy to be minimized. In this case, as mentioned above, \cite{Hal2014} showed that the particle locations $x_n$ satisfy (\ref{e1})
and \cite{McKeague2015} showed that the empirical distribution tends to a standard Gaussian distribution.

%
Our contribution in the present article is to derive an interworld potential for the second  energy state ($k=1$) and show that the empirical distribution of the configuration that minimizes the corresponding Hamiltonian has a limit distribution that again agrees with quantum theory.
The  interworld  potential in this case is shown to be
\begin{equation}
\label{U1}
U_1({\bf x}) = 9\sum_{n=1}^N\left({1\over x_{n+1}^3-x_n^3} -{1\over x_{n}^3-x_{n-1}^3}\right)^2x_n^4
\end{equation}
and the  minimizer of the corresponding Hamiltonian
 ${H_1}({\bf x},{\bf p}) =  E({\bf p}) +V({\bf x}) + U_1({\bf x})$
 is shown to satisfy the recursion
\begin{equation}
\label{rec-max0}
x_{n+1}^3={x_n^3}-3\left(\sum_{i=1}^n \frac{1}{x_i}\right)^{-1}.
\end{equation}
Further, we show  that if $x_1, \ldots, x_N$ is a decreasing, zero-mean solution, then  the empirical distribution of the $x_n$  converges to the (two-sided) Maxwell distribution having density $p_1(x)=x^2 e^{-x^2/2}/\sqrt{2\pi}$.   The entire sequence $x_1, \ldots x_N$ should be viewed as indexed by $N$, though we suppress notation for this dependence and write $x_1, \ldots x_N$ instead of $x_{1,N}, \ldots , x_{N,N}$.  We also give a rate of convergence using a new extension of Stein's method.   Our approach is  generalizable to  recursions that converge to the distributions of other higher energy states of the quantum harmonic oscillator, although we do not pursue such extensions here.

We initially thought that the MIW interpretation could be based on a ``universal"  interworld potential function $U_0$ that applies to all energy levels, with the densities $p_k(x)$  then arising as limits of {\it  local} minima of $H_0$.  However, this idea turned out to be analytically unworkable. Here we propose an alternative  approach   in terms of adapting the interworld potential to each higher energy level.  Minimizing the resulting Hamiltonian is then tractable and the solution can be shown to converge to  $p_k(x)$, at least in the case $k=1$.  
\cite{Hal2014} derived their interworld potential $U_0$ as a discretization of  Bohm's quantum potential  summed over the particle ensemble, see  \cite{Bohm}.  The challenge  in general is to extend this derivation to higher-energy  wave functions in a way that leads to an explicit recursion  minimizing the resulting  Hamiltonian, and to show that it agrees with $p_k(x)$ in the limit.   A major contribution here, in addition to providing a rate of convergence, is a general method for finding such interworld potential functions and their associated particle  recursions.

Stein's method (see \cite{Stein1986}, \cite{che2010} and \cite{Ross2011}) is a well established technique for obtaining explicit error bounds for distributional limit theorems.  The usual ``density approach'' (see \cite{Chatterjee2011a}) for applying Stein's method to arbitrary random variables does not seem to apply in cases where the density function vanishes at a point that is not at the endpoints of the range of the random variable (here we have $p_1(0)=0$ and the random variable can take both negative and positive values); in this case the solution to the Stein equation will have a singularity and also unbounded derivatives, and this therefore requires the new technique we give here to handle such distributions. It is interesting to note that, while there are plenty of examples of Stein's method applied to distributions with a density having a zero at the endpoints of the range of the random variable (the gamma and beta distributions, for example), there have been no examples that we know of where there are zeros inside the range; the higher energy distributions $p_k(x)$, for $k>0$, appear to be the first such distributions considered.  The price one has to pay with our approach for handling these zeros is that more complicated estimates must be made from the couplings. In this case we have an explicit representation of the recursion, and therefore the coupling, and this leads to the possibility of making these estimates.

In Section 2 we generalize the argument of \cite{Hal2014} to derive the interworld potential, and show how it leads to the solution  (\ref{rec-max0}). In Section 3 we introduce the notion of a generalized zero-bias transformation, and show that the distributional properties of eigenstates of the quantum harmonic oscillator can be characterized in terms of fixed points of this transformation.  Also, we derive the generalized zero-bias distribution for the empirical distribution of  general configurations.  Section \ref{stein} develops our results based on the new extension of Stein's method to show convergence of the  configuration that minimizes the Hamiltonian of  the second energy state.

\section{Interworld potentials for higher energy states}
\label{canon}
 \cite{Hal2014}  introduced their MIW theory from the perspective of the de Broglie--Bohm interpretation of quantum mechanics, which is mathematically equivalent to standard quantum theory.  They used this approach to construct an ansatz for the conjectured interworld potential $U_0$  governing the ground state wave function of the quantum harmonic oscillator.  In this section we introduce an extended version of this ansatz aimed at providing a MIW characterization of the  higher energy eigenstates.

 Our argument follows along the lines of Section IIIA of \cite{Hal2014} with the major difference being that we now need to  introduce a more general way of approximating  the density of  particle location for a  stationary wave function $\psi(x)$, namely for a density of the form $p(x) = |\psi(x)|^2= b(x)\varphi(x)$, where $b(x)$ is  a non-negative, even, smooth function having finitely many zeros.  Here $b$ represents a  ``baseline" that varies more rapidly than $\varphi(x)$. 
Let  $x_1>x_2>\ldots > x_N$.   Bohm's quantum potential summed over the  ensemble $\{ x_n\}$ is defined by 
\begin{equation}
U_{\psi}({\bf x})= \sum_{n=1}^N \left[ {p'(x_n)/p(x_n)}\right]^2 \label{qpot}
\end{equation}
where we are using dimensionless units.
 An approximation to $p(x_n)$  based on ignoring $\varphi(x)$ is given (up to a normalizing constant) by
 $$\tilde p(x_n) = {b(x_n)\over B(x_n) -B(x_{n+1})},$$
where $B(x)=\int_0^x b(t)\, dt$ is the cumulative baseline function. This suggests \begin{eqnarray*}
{p'(x_n)\over p(x_n)}&\thickapprox& {\tilde p(x_{n})-\tilde p(x_{n-1})\over (x_{n}-x_{n-1})\tilde p(x_n)}\thickapprox \left[ {1\over B(x_n) -B(x_{n+1})} -{1\over B(x_{n-1})- B(x_{n})}\right] b(x_n),
\end{eqnarray*}
 where we set $B(x_0)=\infty$ and $B(x_{N+1})=-\infty$.
Our proposed ansatz for the interworld potential is then based on inserting the above expression into  (\ref{qpot}) to obtain 
  \begin{equation}
\label{potential}  U_b({\bf x}) = \sum_{n=1}^N\left[ {1\over B(x_{n+1})- B(x_n)} -{1\over B(x_{n})- B(x_{n-1})}\right]^2 b(x_n)^2.
\end{equation}
Note that our earlier assumptions about $b$ imply that $B$ is strictly increasing, so $U_b$ is well-defined.
In the simplest cases  $b(x)=1$ and  $b(x)=x^2$ the above expression for $U_b$ agrees with the interworld potentials $U_0$ and $U_1$  defined in the Introduction.
\medskip

 Specializing to the case $b(x)=x^2$, the following argument characterizes the minimizer of the Hamiltonian $H_1$ (i.e., the ground state when the  interworld potential  is $U_1$) in terms of a solution to the recursion (\ref{rec-max0}).  In any ground state the particles do not move, so the kinetic energy  $E$ vanishes.  Then, adapting the argument of  \cite{Hal2014} to apply to $H_1$, we have
\begin{eqnarray*}
9(N-1)^2
&=& 9\left[ \sum_{n=1}^{N-1} {  x_{n+1}^3-x_n^3 \over x_{n+1}^3-x_n^3}\right]^2 \\
&=& 9\left[  \sum_{n=1}^{N} \left(  {1\over  x_{n+1}^3-x_n^3}- {1\over  x_{n}^3-x_{n-1}^3} \right) x_n^2(x_n-\overline{x^3_N}/x_n^2)  \right]^2\\
&\le&  9\left[\sum_{n=1}^N\left({1\over x_{n+1}^3-x_n^3} -{1\over x_{n}^3-x_{n-1}^3}\right)^2x_n^4\right]\left[ \sum_{n=1}^{N} (x_n-\overline{x^3_N}/x_n^2)^2\right]\\
& \le& U_1({\bf x})\, V({\bf x}), \end{eqnarray*}
where the first inequality is Cauchy--Schwarz.
So $U_1\ge 9(N-1)^2/V$, leading to
$$H_1=U_1+V \ge 9(N-1)^2/V +V \ge 6(N-1)$$
with the last inequality being equality for $V= 3(N-1)$.  It follows that $H_1$ is minimized when $V= 3(N-1)$, the mean $\overline{x^3_N}$ of  $\{x_n^3, n=1,\ldots, N\}$ vanishes, and $${1\over x_n}= \alpha \left[{1 \over x_{n+1}^3-x_n^3} -{1\over x_{n}^3-x_{n-1}^3}\right] $$
for some constant $\alpha$.
The sum of the right of the above display telescopes, leading to the recursion (\ref{rec-max0})
by rearranging and noting that $\alpha =-V/(N-1) =-3$.

 \begin{figure}[!ht]
\begin{center}
  \includegraphics[scale=.43]{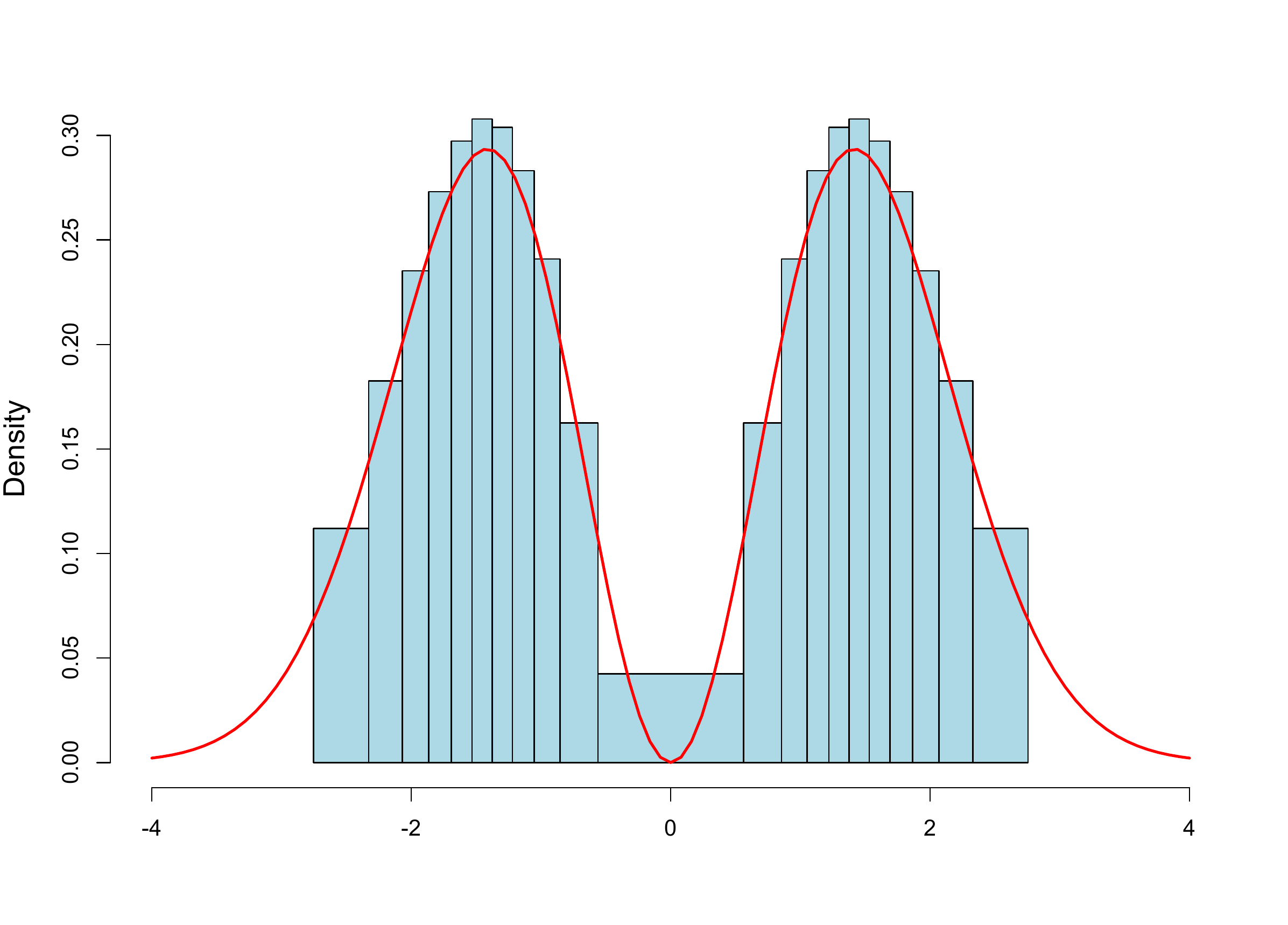}
 \caption{Example with $b(x)=x^2$, $N=22$, showing the piecewise constant density having mass $1/(N-1) $ uniformly distributed over the intervals between successive $x_n$ compared with the Maxwell density, where the breaks in the histogram are the successive $x_n$ satisfying the recursion  (\ref{rec-max0}).}
   \label{hist1}
 \end{center}
   \end{figure}
 \medskip

The following lemma provides the basic properties we need to ensure the existence of a solution of the Maxwell recursion  (\ref{rec-max0}) that minimizes the Hamiltonian $H_1$, as well as ensuring that the solution is unique. This result is analogous to Lemma 1 of \cite{McKeague2015} concerning solutions of (\ref{e1}), but  the difference here  is that the  variance  is 3, agreeing with the Maxwell distribution (rather than close to standard normal  in the case of  (\ref{e1})).
\begin{lemma}\label{lem1} Suppose $N$ is even.
Every zero-median solution $x_1,\ldots, x_N$ of  (\ref{rec-max0}) satisfies:
\begin{itemize}
\item[]
\begin{itemize}
\item[{\rm(P1)}] Zero-mean: $x_1+\ldots +x_N=0$.
\item[{\rm (P2)}] Maxwell variance:   $x_1^2+\ldots +x_N^2=3(N-1)$.
\item[{\rm(P3)}]  Symmetry:  $x_n=-x_{N+1-n}$ for $n=1,\ldots, N$.
\end{itemize}
\end{itemize}
Further, there exists a unique solution $x_1,\ldots, x_N$ such that (P1) and
\begin{itemize}
\item[]
\begin{itemize}
\item[{\rm(P4)}]  Strictly decreasing:  $x_1>\ldots >x_N$
\end{itemize}
\end{itemize}
hold.  This solution has the zero-median property, and thus also satisfies (P2) and (P3).
\end{lemma}
\begin{proof}
The proof follows identical steps to the proof of  Lemma 1 of \cite{McKeague2015}, apart from the variance property (P2), which is proved using (P1) and (P3) as follows.  Denote $S_n= \sum_{i=1}^n x_i^{-1}$ for $n=1,\ldots, N$, and set $S_0=0$.  Using  (\ref{rec-max0}) we can write
\begin{eqnarray*}
3(N-1)&=&3 \overset{N-1}{\underset{n=1}{\sum }}S_nS_n^{-1}=\overset{N-1}{\underset{n=1}{\sum
}}S_n(x_n^3-x_{n+1}^3)=\overset{N-1}{\underset{n=1}{\sum }}[(S_{n-1}+x_n^{-1})x_n^3-S_nx_{n+1}^3]\\
&=&\overset{N-1}{\underset{n=1}{\sum
}}[S_{n-1}x_n^3-S_nx_{n+1}^3+x_n^2]\\
&=& x_1^2+\ldots
+x_{N-1}^2-S_{N-1}x_N^3,
\end{eqnarray*}
where we used the recursion in the second equality, and the last equality is from a telescoping sum.   (P3) implies $S_N=0$, so $-S_{N-1}=1/x_N$, and (P2) follows.
\end{proof}
\medskip

Although  in the sequel we concentrate on the case $k=1$ (see Figure \ref{hist1}), to conclude this section we briefly discuss general densities of the form $p_k$ given in  (\ref{eq:19}).  
The above argument for $b(x)=x^2$ can be extended to general $U_b$ under the condition that   $B(x)$ is proportional to $xb(x)$, which is the case when $b(x)$ is proportional to $x^r$ for some even non-negative integer $r$ (but not for the square of the  $k$th Hermite polynomial  unless $k=0$ or 1).  Under this condition, it can be shown that the minimizer of the Hamiltonian based on $U_b$ is a symmetric solution of the recursion
\begin{equation}
\label{genrec}
 B(x_{n+1})=  B(x_{n})- \left(\sum_{i=1}^n \frac{x_i}{b(x_i)}\right)^{-1}.
\end{equation}
We have not been able to show that this recursion minimizes the Hamiltonian for general $b$, but our numerical results suggest that it is very close if not identical to a minimizer.  With  $k=2$ we have $b(x)=(x^2-1)^2/2$, $B(x) = x^5/10-x^3/3+x/2$, and
the symmetric solution of the resulting recursion
produces a remarkably good agreement with $p_k$, see Figure \ref{Herm2}.

 \begin{figure}[!ht]
\begin{center}
  \includegraphics[scale=.52]{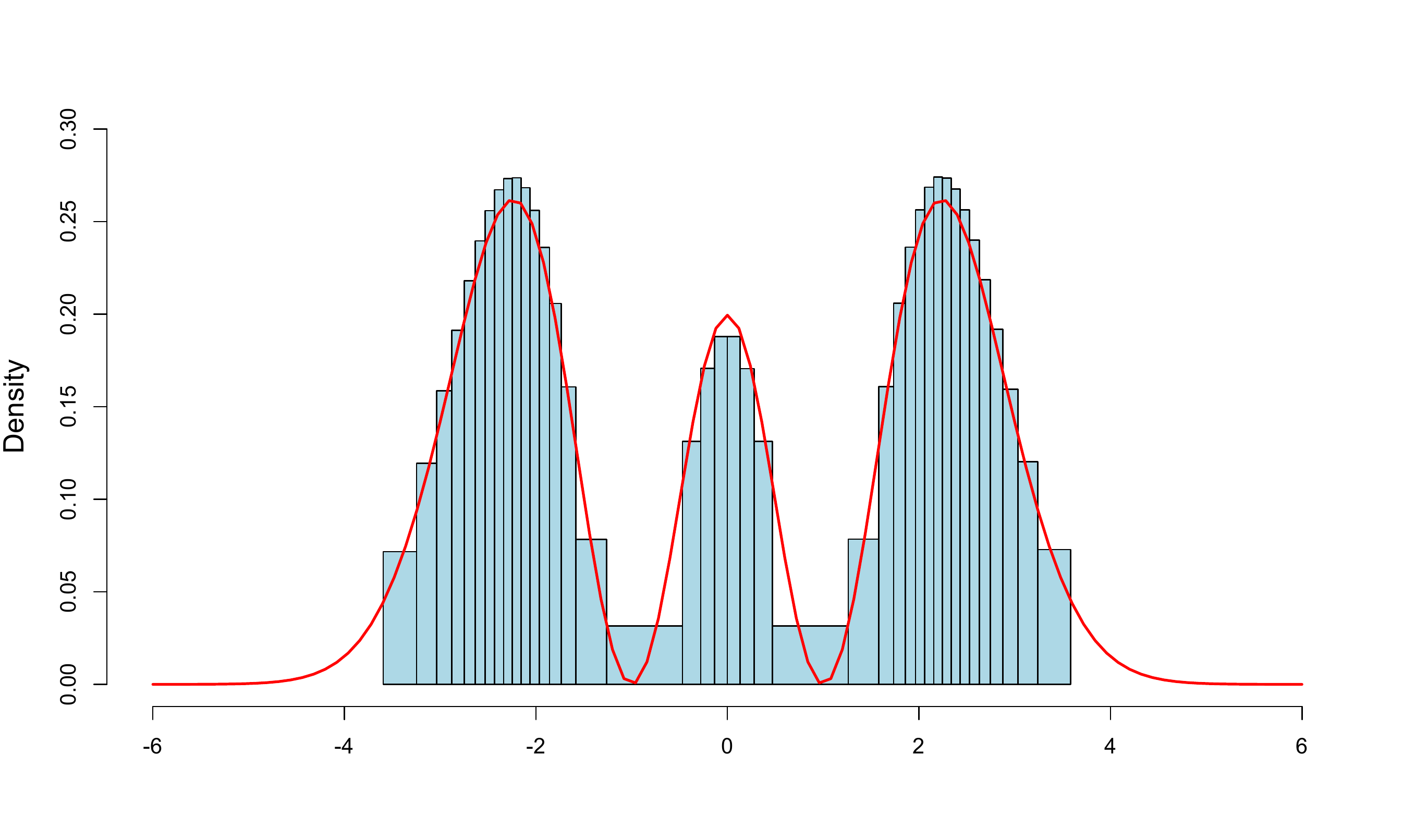}
\caption{Example with $b(x)={\rm He}_k(x)^2/k!$ for $k=2$, $N=41$,  where the breaks in the histogram are the successive $x_n$ satisfying the recursion  (\ref{genrec}) and the red curve is $p_k(x)$.}
\label{Herm2}
\end{center}
 \end{figure}
 \medskip

\section{Generalized zero-bias transformations}
\label{gzbds}
Let $W$ be a symmetric random variable and
$b\colon \mathbb{R} \to \mathbb{R} $ a non-negative function such that
$\sigma^2=\IE[W^2/b(W)]<\infty$. \cite{Goldstein1997} gives a distributional fixed point characterization of the Gaussian distribution, which we generalize in the definition below.
\begin{definition} \label{d1}   If there is a random variable\ $W^{\star}$ such that $$\sigma^2  \IE\left[\frac{f'(W^{\star})}{b(W^{\star})}\right]=\IE\left[\frac{Wf(W)}{b(W)}\right]$$
for all  absolutely continuous functions $f\colon \mathbb{R}\to
\mathbb{R}$ such that $\IE|Wf(W)/b(W)|<\infty$, we say that $W^{\star}$ has
the {\it $b$-generalized-zero-bias distribution} of $W$.
\end{definition}
\begin{remark}
\cite{Goldstein1997} study the case $b(x)=1$ and show that $W^{\star}$ has the same distribution as $W$ if and only if $W$ has a Gaussian distribution.  Distributional fixed point characterizations for exponential, gamma and other nonnegative distributions and the connection with Stein's method have been studied in \cite{Pekoz2011}, \cite{Pekoz2013}, and \cite{Pekoz2015}.
\end{remark}

\begin{remark}\label{remark3.3}
  By a routine extension of the proof of Proposition 2.1 of
  \cite{che2010}, it can be shown that there exists a unique
  distribution for $W^{\star}$, and it is absolutely continuous with
  density
  $$p^{\star}(x)\propto {b(x)}\IE\left[\frac{W}{b(W)}1_{W\geq
      x}\right].$$
  We note in passing that the $\sigma^2$ is misplaced in the first
  display of Chen et al.'s proposition, which corresponds to $b(x)=1$,
  the usual zero-bias distribution of $W$.  The composition of the
  $b$-generalized-zero-bias transformation with the
  $(1/b)$-generalized-zero-bias transformation is the usual zero-bias
  transformation.
\end{remark}

\begin{remark}
With $\varphi$ the standard normal density and $b$ a $\varphi$-integrable function,
if $W$ has density
 \begin{equation}\label{eq:13}
   p(x)=b(x) \varphi(x),
 \end{equation}
 then its distribution is a fixed point for the
 $b$-generalized-zero-bias
 transformation since
  $$p^{\star}(x)=b(x)\int_x^\infty \frac{t}{b(t)}p(t)\, dt=
  b(x)\int_x^\infty t\varphi(t)\, dt=p(x). $$
\end{remark}

The following  result gives the $b$-generalized-zero-bias distribution of the uniform distribution on $N$ points.

   \begin{proposition} \label{prop:gener-zero-bias}

Given an integer $N>1$, let $x_1>x_2>\ldots > x_N$  be such that $b(x_n)>0$ for all $n$.
Let $\mathbb{P}_N$ be  the  empirical distribution of the $x_n$: $$\mathbb{P}_N(A)={\#\{n\colon x_n\in A\}\over N}$$  
for any Borel set $A\subset \mathbb{R}$. Under the symmetry condition $x_n=x_{N-n+1}$ for  $n=1,\ldots, N$,  the $b$-generalized-zero-bias distribution ${\mathbb{P}}^{\star}_N$ of $\mathbb{P}_N$ is defined, and has density
$$p^{\star}(x)\propto  b(x)\left[\sum_{i=1}^n \frac{x_i}{b(x_i)}\right] $$     
for  $x_{n+1}<  x \le x_{n}$ ($n=1,\ldots, N-1$), and $p^{\star}(x) = 0$ if $x> x_1$ or  $x\le x_N$.  \end{proposition}
\begin{proof}
  Immediate from Remark \ref{remark3.3}.  \end{proof}

Recall the following distances between distribution functions  $F$ and $G$.  The Kolmogorov distance is
$$\dk(F,G)=\sup_{x\in \IR}|F(x)-G(x)|,$$
and the  Wasserstein distance  is
$$\dw(F,G)=\sup_{h\in \mathcal{H}}\left|\int_\IR h\, dF-\int_\IR h\, dG\right|$$
where $$\mathcal{H}=\{h\colon \IR\rightarrow \IR \mbox{ Lipschitz with } \| h'\|\leq 1\}$$
and $\|\cdot \|$ is the supremum norm.  Using Proposition 1.2 in \cite{Ross2011}, these two metrics are seen to be related by
$$\dk(F,G)\leq \sqrt{2C\dw(F,G)}$$ if $G$ has density bounded by $C$.

Restricting attention to the  special case $b(x)=x^2$, we can now state our main result, along with an important  corollary.
\begin{theorem}\label{maintheorem}
Suppose $W^{\star}$ is constructed on the same probability space as the zero-mean random variable $W$ and is distributed according to the $x^2$-generalized-zero-bias distribution of $W$.  Let $M$ have the two-sided Maxwell density $x^2e^{-x^2/2}/\sqrt{2\pi}$.  Then there exist positive
finite constants $\lambda_1, \lambda_2, \lambda_3$ and $\lambda_4$ such that
  \begin{align}
     \dw(\law(W),\law(M)) & \le
     \lambda_1 \mathbb{E} \left| W - W^{\star} \right| + \lambda_2
     \mathbb{E}\left[ |W| \left| W - W^{\star} \right|
     \right]\nonumber \\
&\label{eq:6} \qquad\quad + \lambda_3 \mathbb{E} \left|
       \frac{1}{W}-\frac{1}{W^{\star}} \right| + \lambda_4 \mathbb{E} \left| 1-\frac{W^{\star}}{W} \right|.
  \end{align}
\end{theorem}
\begin{proof}
  The inequality follows immediately from Theorem \ref{maxthm}. Finiteness of the constants (along with explicit upper bounds)  is
  detailed in Proposition~\ref{prop:boundonlambdas}.
\end{proof}

The following corollary gives a rate of convergence of the solution to \eq{rec-max0} to the two-sided Maxwell distribution in terms of the Wasserstein distance;  we postpone the proof until  Section 4.3.  
\begin{corollary}\label{maincorollary}
Suppose  $x_1, \ldots x_N$ is a monotonic, zero-mean, finite sequence of real numbers satisfying \eq{rec-max0}, let $\mathbb{P}_N$ be the empirical distribution of these values, and let $M$ be as in Theorem \ref{maintheorem}. Then there is a constant $c>0$ such that
$$\dw(\mathbb{P}_N,\law(M))  \le  c\sqrt{\frac{\log N}{N}}.$$
\end{corollary}

\section{The Stein equation and its solutions}
\label{stein}
\subsection{General considerations}
Let $X$ have a density {$p$ defined} as in \eqref{eq:13}. The first step
is to identify an appropriate ``Stein equation'' and bound its
solutions.  Let $\tilde{h}$ be such that $\IE[\tilde{h}(X)] =
0$.
There are many possible starting points.  The well known ``density
approach'' (see \cite{Chatterjee2011a}) starts with the Stein equation
 \begin{equation*}
   f'(x)+\frac{p'(x)}{p(x)}f(x) = \tilde{h}(x)
 \end{equation*}
 which is easily solved to yield
\begin{equation*}
   f(x) =  \frac{1}{b(x)\varphi(x)}
   \int_{x}^\infty \tilde{h}(u)b(u) \varphi(u)\,du.
 \end{equation*}
 If $b$ vanishes at a point, as it does in the two-sided Maxwell case
 where $b(0)=0$, then this solution $f$ will have a singularity and we
 cannot carry on with the usual program for applying Stein's
 method. For these types of distributions we propose a new approach;
 the price we pay here is the necessity to bound several additional
 quantities concerning the couplings we obtain.  The explicit nature
 of the recursion here allows us to compute these quantities.

 In view of Definition~\ref{d1} it is natural to consider the Stein
 equation
\begin{equation}\label{eq:15}
  \frac{f'(w)}{b(w)}-w \frac{f(w)}{b(w)}=\tilde{h}(w)
\end{equation}
which can be solved using the usual normal approximation solution, but
the resulting estimates will again rest on properties of $f/b$ which
is unbounded in the cases we are interested in.  Because of this, we
introduce an original route which leads to the correct order bounds we
are seeking.

First, following \cite{LRS16} we introduce an integral operator
associated with the Gaussian density:
\begin{equation}\label{eq:invstop}
\tilde{h} \mapsto   \mathcal{T}_{\varphi}^{-1} \tilde{h}(w)  :=
\frac{1}{\varphi(w)} \int_w^{\infty} \tilde{h}(u) \varphi(u) \,du
\end{equation}
(called the ``inverse Stein  operator'') mapping functions
with Gaussian-mean zero and sufficiently well-behaved tails into
bounded functions. Next define the ``Stein kernel'' of $X$ (or, equivalently, of $p$) by
\begin{equation*}
  \tau_X(x) = \frac{1}{p(x)} \int_x^{\infty} u p(u) \,du.
\end{equation*}
  The Stein kernel satisfies the integration by parts formula
\begin{equation}
   \IE \left[ \tau_X(X)f'(X)  \right] = \IE \left[ X f(X) \right]
\end{equation}
for all sufficiently regular functions $f$.  
{Stein kernels were introduced in \cite{Stein1986}, \cite{Cacoullos1989}, and
  have proven to be of great use in Gaussian analysis, see, e.g.,
  \cite{NP09} and \cite{C09}. 
 \begin{remark}
   As discussed in the Introduction, our  concern in this paper is
   with symmetric densities of the form
   $p_k(x)$ given in (\ref{eq:19}).  
 The   Stein kernel of  $p_k$ is
$$\tau_k(x) =  \int_x^{\infty} u {\rm He}_k(u)^2
\varphi(u)\,du/(\varphi(x) {\rm He}_k(x)^2)
$$
and direct integration leads to
  \begin{equation}\label{eq:27}
    \tau_1(x) = \frac{x^2+2}{x^2},\  \tau_2(x) = \frac{x^4+2x^2+5}{(x^2-1)^2},\
    \tau_3(x) = \frac{x^6+9x^2+18}{(x^3-3x)^2}.
  \end{equation}
The next result applies to general densities, but in the following sections we focus on further results for the special case $\tau_1(x). $
 \end{remark}
We aim to assess
the proximity between the law of $X$ and some $W$ by estimating the Wasserstein distance between their distributions.
Suppose that there exists a $W^{\star}$ following the
$b$-generalized-zero-bias distribution of $W$ and defined on the same probability
space as $W$. To each integrable test function $h$ we associate the function
$f:=f_h$, the solution to \eqref{eq:15} with $\tilde{h} = h - \IE h(X)$.
This association is unique in the sense that there exists only one
absolutely continuous version of $f$ satisfying \eqref{eq:15} at all
points $x$, see e.g.\ \cite{che2010}.  We then write (supposing here
and in the sequel that $\sigma^2=1$)
\begin{align}
&  \IE \left[ h(W) \right]-\IE \left[ h(X)
  \right]  = \IE\left[  \frac{f'(W)}{b(W)}-W \frac{f(W)}{b(W)}
  \right] \nonumber \\
& = \IE \left[ \frac{f'(W)}{b(W)} -
\frac{f'(W^{\star})}{b(W^{\star})} \right]. \label{eq:16}
\end{align}
One of the major differences
  between the present setting and the usual applications of Stein's
  method is that here we cannot bound the right
  hand side of \eqref{eq:16} directly because
  solutions $f$ of \eqref{eq:15} have a singularity at 0. In order to
  bypass this difficulty it was therefore necessary to introduce some
  further tweaking of the method which we now detail. Such results should also have an intrinsic
  interest for users of Stein's method as it bears natural
  generalizations for density functions of the form $p(x) = b(x)
  p_0(x)$ with $p_0$ and $b$ well chosen.

\begin{proposition}
  Let $x \mapsto b(x)$ be a nonnegative even function with support in
  $(-\infty, \infty)$ such that
  $\lim_{x\to \pm \infty} b(x) \varphi(x) = 0$. Suppose furthermore
  that $b$ is absolutely continuous and integrable w.r.t.\ $\varphi$
  with integral $\int_{-\infty}^{\infty}b(x) \varphi(x) \,dx = 1$. Let
  $X$ be a random variable with density $x \mapsto b(x)
  \varphi(x)$. Then
\begin{equation}\label{eq:28}
   \tau_X(x) = 1 + \frac{\mathcal{T}_{\varphi}^{-1}b'(x)}{b(x)}
\end{equation}
under the convention that the ratio is set to zero at all points $x$ such
that $b(x) = 0$ and $\mathcal{T}_{\varphi}^{-1}b'(x) \neq 0$.
Further, with  $\tilde{h}$ defined by \eqref{eq:15},
\begin{equation}
  \label{eq:14}
g_h(x) = \frac{{\int_x^{\infty} b(u) \tilde{h}(u) \varphi(u)
 \,du}}{{b(x) \varphi(x)+\int_x^{\infty} b'(u) \varphi(u)\,du}  }
\end{equation}
is the unique bounded solution of the ODE
\begin{equation}
  \label{eq:1}
\tau_X(x) g'(x) - xg(x) = \tilde{h}.
\end{equation}
\end{proposition}

\begin{proof}
{
Integrating by parts in the definition of the Stein kernel for $p =
b\varphi$ we get  (assuming that  $\lim_{x\to \pm \infty} b(x) \varphi(x) =
0$)
 \begin{align*}
    \int_x^{+\infty}yp(y)\,dy & = \int_x^{\infty} b(y)
                                (-\varphi'(y))\,dy \\
   & = b(x) \varphi(x) + \int_x^{\infty}
     b'(y) \varphi(y) dy
  \end{align*}
  so that \eqref{eq:28} follows by definition \eqref{eq:invstop} of
  the inverse Stein operator. For the second claim note how
  \begin{equation*}
    \tau_X(x) g'(x) - x g(x) = \frac{\left( \tau_X(x) g(x) p(x) \right)'}{p(x)}
  \end{equation*}
so that the conclusion (along with unicity) follows from the same
argument as in \cite[Proposition 3.2.2]{NP12}
 }
\end{proof}


Intuition (supported e.g.\ by \cite[Lesson VI]{Stein1986} or the more
recent work \cite{Dobler12}) encourages us to claim that functions
\eqref{eq:14} will have satisfactory behavior.  It is thus natural to
seek a connection between equations of the form \eqref{eq:15} and
\eqref{eq:1}. {This we summarize in the next lemma.
\begin{lemma}\label{lem:rewritingsteq}
  Let all notations be as above and introduce the function $g=g_f$
  defined at all $x$ through
\begin{equation*}
  \frac{f'(x)-xf(x)}{b(x)} = \tau_X(x) g'(x) - xg(x).
\end{equation*}
Then
\begin{align}
\IE \left[ \frac{f'(W)}{b(W)} -
\frac{f'(W^{\star})}{b(W^{\star})} \right]   &  = 
\IE \left[ W (\tau_X(W) - 1) g(W) - W^{\star} (\tau_X(W^{\star}) - 1)
   g(W^{\star})  \right]\nonumber\\
& \quad \quad + \IE \left[ \tau_X(W) g'(W) -
   \tau_X(W^{\star}) g'(W^{\star}) \right].
   \label{eq:5}
\end{align}
\end{lemma}
\begin{proof}
Since
\begin{equation*}
  \frac{f'(x)-xf(x)}{b(x)} =  \frac{(f(x)\varphi(x))'}{b(x)\varphi(x)}
\end{equation*}
 and
 \begin{equation*}
   \tau_X(x) g'(x) - xg(x)
= \frac{(b(x) \tau_X(x) g(x)\varphi(x))'}{b(x) \varphi(x)}
 \end{equation*}
 at all $x$ for which $b(x) \neq 0$, we deduce that $f$ and $g$ are
mutually defined by
$  f = (b\tau_X)g
$.  This in turn gives
\begin{equation*}
  \frac{f'(x)}{b(x)}  = \left( \frac{b'(x)}{b(x)} \tau_X(x) +
                        \tau_X'(x) \right)g(x) + \tau_X(x) g'(x) =:
                      \psi(x) g(x) + \tau_X(x)g'(x)
\end{equation*}
which, combined with
$  \psi(x) = x(\tau_X(x)-1)
$ (that is easily derived using the various definitions involved),
leads to the useful identity
\begin{equation}
  \label{eq:18}
 \frac{f'(x)}{b(x)} =  x (\tau_X(x) - 1) g(x) + \tau_X(x) g'(x)
\end{equation}
from which \eqref{eq:5} is directly derived.
\end{proof}
In view of Lemma~\ref{lem:rewritingsteq} and identity \eqref{eq:16} it
 remains to find bounds on the two terms on the right hand side  of
\eqref{eq:5}.
}

\subsection{Approximating the two-sided Maxwell distribution}

\begin{theorem}
\label{maxthm}
Let $p(x) = x^2 \varphi(x)$, and take $f$ a solution to the Stein
equation 
\begin{equation}\label{eq:4}
  f'(w)/w^2 - wf(w) / w^2 = \tilde{h}(w),
\end{equation}
where $\tilde{h}$ is a function having bounded first derivative and
zero-mean under $p$. Set $c = \| \tilde{h}'\|$.  Then for any coupling
of $W$ and $W^*$ on a joint probability space such that $W^*$ has the $x^2$-generalized zero biased distribution for $W$,
  \begin{align}
     \left| \mathbb{E} \left[\frac{f'(W)}{(W)^2}  -
         \frac{f'(W^{\star})}{(W^{\star})^2} \right] \right| & \le
     \lambda_1 \mathbb{E} \left| W - W^{\star} \right| + \lambda_2
     \mathbb{E}\left[ |W| \left| W - W^{\star} \right|
     \right]\nonumber \\
&\label{eq:6} \qquad\quad + \lambda_3 \mathbb{E} \left|
       \frac{1}{W}-\frac{1}{W^{\star}} \right| + \lambda_4 \mathbb{E} \left| 1-\frac{W^{\star}}{W} \right|
  \end{align}
with
{
  \begin{align}\label{eq:23}
   & \lambda_1  \le   6c, \,  \lambda_2\le 7c, \, 
    \lambda_3\le  18c \mbox{ and }  \lambda_4\le 22c. 
  \end{align}
}
\end{theorem}

\begin{proof} With $b(x) =x^2$ we have
  $\tau_X(x) = 1+2/x^2$ and $\psi(x) = 2/x$,  so that   \eqref{eq:5}
  becomes
\begin{align*}
 & =   \mathbb{E} \left[ \frac{2}{W^{\star}}
                                               g(W^{\star}) - \frac{2}{W}
                                               g(W) \right] + \mathbb{E} \left[
                                               \left( 1+\frac{2}{(W^{\star})^2} \right)
                                               g'(W^{\star}) - 
 { \left(  1+\frac{2}{(W)^2} \right)}     g'(W)\right]\\
& = 2\mathbb{E} \left[ \left(  \frac{1}{W^{\star}}-\frac{1}{W}  \right)
  g(W^{\star})\right]
+ 2\mathbb{E} \left[ \frac{1}{W} \left( g(W^{\star}) - g(W) \right)
  \right] \\
  & \quad + \mathbb{E} \left[ g'(W^{\star})-g'(W) \right]  + 2\mathbb{E} \left[ \frac{1}{(W^{\star})^2} g'(W^{\star})  -
    \frac{1}{(W)^2} g'(W)\right].
\end{align*}
The first two terms are dealt with  easily to get
\begin{align*}
&  2\left| \mathbb{E} \left[ \left(  \frac{1}{W^{\star}}-\frac{1}{W}  \right)
  g(W^{\star})\right]
+ 2\mathbb{E} \left[ \frac{1}{W} \left( g(W^{\star}) - g(W) \right)
  \right]\right| \\
& \le 2\|g\|  \mathbb{E} \left[ \left|
  \frac{1}{W^{\star}}-\frac{1}{W}  \right| \right] +  2\|g'\|
  \mathbb{E} \left[ \frac{1}{|W|} \left| W^{\star}-W \right| \right].
\end{align*}
For the last two terms we introduce the function
\begin{align*}
  \chi(x) = g'(x)/x
\end{align*}
to get on the one hand
\begin{align*}
   \mathbb{E} \left[ g'(W^{\star})-g'(W) \right]  &  =   \mathbb{E}
                                                    \left[
                                                    W^{\star}
                                                    \frac{g'(W^{\star})}{W^{\star}}-W
                                                    \frac{g'(W)}{W}
                                                    \right] \\
&=  \mathbb{E}
                                                    \left[(W^{\star}-W)
                                                    \chi(W^{\star})\right]
                                                    + \mathbb{E}
                                                    \left[ W \left(
                                                    \chi(W^{\star}) - \chi(W) \right) \right]
\end{align*}
so that
\begin{align*}
  \left|  \mathbb{E} \left[ g'(W^{\star})-g'(W) \right]  \right|\le  \| \chi\| E
  \left[ \left| W^{\star}-W \right| \right] + \|\chi'\| E \left[ \left| W(W^{\star}-W) \right| \right]
\end{align*}
and, on the other hand
\begin{align*}
  \mathbb{E} \left[ \frac{1}{(W^{\star})^2} g'(W^{\star})  -
    \frac{1}{(W)^2} g'(W)\right] & = \mathbb{E} \left[ \frac{1}{W^{\star}} \chi(W^{\star})  -
    \frac{1}{W} \chi (W)\right] \\
& = \mathbb{E}  \left[ \left(  \frac{1}{W^{\star}} - \frac{1}{W}
  \right)  \chi(W^{\star})\right] + \mathbb{E} \left[  \frac{1}{W}\left(
  \chi(W^{\star})-\chi(W)  \right)  \right]
\end{align*}
so that
\begin{align*}
 & 2 \left|  \mathbb{E} \left[ \frac{1}{(W^{\star})^2} g'(W^{\star})  -
    \frac{1}{(W)^2} g'(W)\right] \right| \le  2\| \chi\| \mathbb{E} \left[ \left|
  \frac{1}{W^{\star}}-\frac{1}{W}  \right| \right] +
2\|\chi'\|\mathbb{E} \left[ \frac{1}{|W|} \left| W^{\star}-W \right|
   \right].
\end{align*}
{Combining these different estimates we obtain \eqref{eq:6}, with
$\lambda_1, \lambda_2, \lambda_3$ and $\lambda_4$ expressed in terms
of $\| \chi \|,\| \chi' \|, \| g \|$ and $\| g' \|$ as follows:
\begin{equation*}
  \lambda_1 = \| \chi \|, \lambda_2 = \| \chi' \|, \lambda_3 = 2(\| g
  \|+\| \chi \|) \mbox{ and }\lambda_4 = 2(\| g' \|+\| \chi' \|).
\end{equation*}
The inequalities in \eqref{eq:23} are proved in the Proposition
\ref{prop:boundonlambdas} below.  }\end{proof}

{The next step is to bound $\| \chi \|,\| \chi' \|, \| g \|$ and
  $\| g' \|$ in a non trivial way; this we achieve in the
  next proposition}.

\begin{proposition}\label{prop:boundonlambdas}
  Let $h:\mathbb{R}\to\mathbb{R}$ be absolutely continuous and
  integrable with respect 
  to $p(x)=x^2 \varphi(x)$. Set 
  $c = \| h'\|$ which we suppose to be finite. Let $X \sim p$,  define
\begin{align}\label{eq:35}
  g_0(x) = \left\{
  \begin{array}{cl}
e^{x^2/2} \int_x^{\infty} y^2 \left( h(y) - \IE \left[
  h(X) \right]  \right)  e^{-y^2/2} \, dy & \mbox{ if }x>0\\
e^{x^2/2} \int_{-\infty}^x y^2 \left( h(y) - E \left[
   h(X)\right]  \right)  e^{-y^2/2} \, dy & \mbox{ if }x\le 0\\
  \end{array}
 \right. 
\end{align}
and set 
\begin{equation}\label{eq:7}
  g(x) = \frac{g_0(x)}{ x^2+2}
  \mbox{ and } \chi(x) = \frac{g'(x)}{x}.
\end{equation} Then
  \begin{align*}
\|g\| \le 3c, 
    \| g' \| \le 4c, 
   \| \chi \| \le   6c
\mbox{ and }
   \| \chi' \| \le 7 c.
  \end{align*}
\end{proposition}

\begin{remark}
The function $g_0$ defined in \eqref{eq:35} satisfies 
\begin{equation}\label{eq:3}
  g_0'(x) - x g_0(x) = x^2\left( h(x) - \IE [Z^2h(Z)] \right)
\end{equation}
with $Z \sim \varphi$ a standard Gaussian random variable.
\end{remark}

\begin{remark}
The function $g$ defined in \eqref{eq:7}  satisfies
\begin{align*}
\frac{\left(  g(x) \tau(x) p(x)\right) '}{p(x)} = h(x) - \IE \left[ h(X) \right]
\end{align*}
with $X \sim p$ and $\tau(x) = 1+2/x^2$.   
\end{remark}


\begin{proof} In order to simplify future notations we introduce
  $ \Phi(x) = \int_{-\infty}^x \varphi(t) \, dt,$
  $ \bar{\Phi}(x) = \int_x^{\infty} \varphi(t) \, dt$
  $ \Upsilon(x) = e^{x^2/2}\int_x^{\infty} t^2 e^{-t^2/2} \, dt$ and
  $\bar{\Upsilon}(x) = e^{x^2/2}\int_{-\infty}^x t^2 e^{-t^2/2} \,
  dt. $ Using the identity
\begin{align}\label{eq:8}
 \int_a^{b} t^2e^{-t^2/2}\,dt = a e^{-a^2/2}-b e^{-b^2/2}+
  \int_a^{b}e^{-t^2/2}\,dt, \quad -\infty \le a < b \le \infty,  
\end{align}
we deduce that
$ \Upsilon(x) = x + e^{x^2/2}\int_x^{\infty}e^{-t^2/2}\,dt$ and
$\bar{\Upsilon}(x) = -x + e^{x^2/2}\int_{-\infty}^xe^{-t^2/2}\,dt $
and thus
\begin{equation}
  \label{eq:9}
  \Upsilon(x),  \bar{\Upsilon}(x) \le |x| + \sqrt{\frac{\pi}{2}} \mbox{ at all } x \in
    \mathbb{R} \mbox{ and }\lim_{x\to
    \infty} \frac{\Upsilon(x)}{x} = \lim_{x\to
    -\infty} \frac{\bar{\Upsilon}(x)}{x}   = 1.
\end{equation}
 The proof is now broken down into several steps. 

\

\noindent \underline{Step 1:} rewrite the solutions. Following  \cite[page
39]{che2010} we rewrite the 
test functions in term of their derivatives (still with $Z$ a standard
normal random variable)
\begin{align*}
 & h(y) - E \left[ h(X) \right]  =    h(y) - E \left[ Z^2h(Z) \right]\\
 & = \int_{-\infty}^{\infty} z^2
                                     \left( h(y) - h(z)
                                     \right)\varphi(z) \, dz \\
  & = \int_{-\infty}^y z^2 \left( \int_z^y h'(t) \, dt \right) \varphi(z)
    \, dz - \int_y^{\infty} z^2 \left( \int_y^z h'(t) \, dt \right)
    \varphi(z) \, dz.
\end{align*}
Changing the order of integration then using \eqref{eq:8}  leads to
the rhs becoming 
\begin{align*}
  &  \int_{-\infty}^y h'(t) \left[ \int_{-\infty}^t
    z^2\varphi(z) \, dz \right] \, du
    - \int_y^{\infty} h'(t) \left[
    \int_t^{\infty} z^2 \varphi(z) \,
    dz\right] \, dt \\
  & =  \int_{-\infty}^y h'(t) \left[  -t \varphi(t) + \int_{-\infty}^t
    \varphi(z) dz \right] \, dt  - \int_y^{\infty} h'(t) \left[ t \varphi(t) + 
    \int_t^{\infty} \varphi(z) \,
    dz\right] \, dt \\
  & = -  \int_{-\infty}^\infty h'(t) t \varphi(t) \, dt +
    \int_{-\infty}^y h'(t) \Phi(t) \, dt -  \int_y^{\infty} h'(t)
    \bar{\Phi}(t) \, dt,
\end{align*}
and thus 
\begin{align}\label{eq:38}
  h(y) - E \left[h(X) \right] = \int_{-\infty}^yh'(t) \Phi(t) dt -
  \int_y^{\infty}h'(t) \bar{\Phi}(t))dt - E \left[ Z h'(Z) \right]. 
\end{align}
We deduce the following useful bound 
\begin{align}
  \label{eq:32}
  \frac{ h(x) - E \left[h(X) \right]}{x^2+2} \le c
  \frac{2(x+1/\sqrt{2\pi}) + \sqrt{2/\pi}}{x^2+2} \le 2c.
\end{align}
Plugging \eqref{eq:38} in \eqref{eq:35} leads to (we restrict the
discussion to $x>0$, the other case following by symmetry)
\begin{align*}
  g_0(x) & = - \IE[Zh'(Z)] \Upsilon(x) & =: I(x)\\
  &  \quad +  e^{x^2/2} \int_x^{\infty} \int_{-\infty}^y y^2
    e^{-y^2/2} h'(t) \Phi(t) \, dt dy  & =: II(x)  \\
  & \quad -   e^{x^2/2} \int_x^{\infty} \int_y^{{\infty}} y^2
    e^{-y^2/2} h'(t) \bar{\Phi}(t) \, dt dy & =: III(x) 
\end{align*}
To deal with  the quantities $II(x)$ and $III(x)$ we again interchange
integrations to get 
\begin{align*}
  II(x) & = e^{x^2/2} \int_{-\infty}^x \left( \int_x^{\infty} y^2e^{-y^2/2}
          \, dy \right) h'(t) \Phi(t) \, dt \\
& \quad + e^{x^2/2} \int_x^{\infty} \left( \int_t^{\infty} y^2e^{-y^2/2}
          \, dy \right)  h'(t) \Phi(t) \, dt \\
  & = \Upsilon(x) \int_{-\infty}^xh'(t) \Phi(t) \, dt + e^{x^2/2}
    \int_x^{\infty} e^{-t^2/2}  \Upsilon(t) h'(t) \Phi(t) \, dt 
\end{align*}
and 
\begin{align*}
  III(x) & = e^{x^2/2}  \int_x^{\infty} \left( \int_x^{t} y^2e^{-y^2/2}
          \, dy \right) h'(t) \bar{\Phi}(t) \, dt \\
  & = e^{x^2/2}  \int_x^{\infty} \left(  e^{-x^2/2}\Upsilon(x) - 
          e^{-t^2/2}\Upsilon(t)  \right) h'(t) \bar{\Phi}(t) \, dt \\
& = \Upsilon(x) \int_x^{\infty} h'(t) \bar{\Phi}(t) \, dt - e^{x^2/2}
  \int_x^{\infty} e^{-t^2/2} \Upsilon(t) h'(t)\bar{\Phi}(t) \, dt
\end{align*}
and thus if $x\ge 0$ we have
\begin{align}
  \label{eq:36}
  g_0(x) & = - \IE[Zh'(Z)] \Upsilon(x) + \Upsilon(x)
           \int_{-\infty}^xh'(t) \Phi(t) \, dt \nonumber\\
& \quad  -\Upsilon(x)
           \int_x^{\infty} h'(t) \bar{\Phi}(t) \, dt + e^{x^2/2}
  \int_x^{\infty} e^{-t^2/2} \Upsilon(t) h'(t) \, dt.
\end{align}
By a similar argument we deduce that if $x < 0$ then 
\begin{align}
\label{eq:31}
  g_0(x) & = - \IE[Zh'(Z)] \Upsilon(x) + \Upsilon(x)
           \int_{-\infty}^xh'(t)\bar{\Phi}(t)  \, dt \nonumber\\
& \quad  -\Upsilon(x)
           \int_x^{\infty} h'(t) \Phi(t)\, dt + e^{x^2/2}
  \int_{-\infty}^x e^{-t^2/2} \Upsilon(t) h'(t) \, dt.
\end{align}

\

\noindent \underline{Step 2:} a bound on $\|g\|$.  Supposing $\|h'\|\le c$ we can
use \eqref{eq:36} and the first claim in \eqref{eq:9} to deduce that
for $x \ge 0$: 
\begin{align*}
  \left| g_0(x) \right| & \le  cE \left| Z \right| (x + \sqrt{\pi/2}) +
                          c(x + \sqrt{\pi/2})\int_{-\infty}^x\Phi(t)\, dt \\
& \quad +
                          c(x + \sqrt{\pi/2})\int_x^{\infty} \bar{\Phi}(t)\,
                          dt + ce^{x^2/2}\int_x^{\infty}e^{-t^2/2}(t + \sqrt{\pi/2})\,dt.
\end{align*}
The last two terms decrease strictly to 0 as $x \to \infty$, with
maximum value $c/2$ and $c(1 + \pi/2)$, respectively.  The first term
is equal to $c (\sqrt{2/\pi}x + 1)$ and the second one  is equal to
\begin{align*}
   c(x + \sqrt{\pi/2})\int_{-\infty}^x\Phi(t)\, dt &  =  c(x +
                                                     \sqrt{\pi/2})
                                                     \left( x \Phi(x)
                                                     + \varphi(x)
                                                     \right) \\
& \le c(x^2+(\sqrt{\pi/2}+1/\sqrt{2\pi})x+1/2).
\end{align*}
 Similar (symmetric) bounds hold for $x \le 0$
and thus, collecting all these estimates, we may
conclude: 
\begin{equation}
  \label{eq:10}
  |g(x)|  = \frac{|g_0(x)|}{x^2+2} \le 3 c. 
\end{equation}

\

\noindent \underline{Step 3:}  a bound on $\|g'\|$.  Here we start by rewriting
the derivative as 
\begin{equation}
  \label{eq:11}
  g'(x) = \frac{g_0'(x)}{x^2+2} - \frac{2x}{(x^2+2)^2} g_0(x). 
\end{equation}
Using \eqref{eq:10}, the second summand is easily seen to be uniformly
bounded (by $3c$). We are left with the first summand for which we
start by rewriting the numerator, for $x \ge 0$, using \eqref{eq:36}:
\begin{align*}
  g_0'(x) & = - \Upsilon'(x) \IE[Zh'(Z)] + \Upsilon'(x)
            \int_{-\infty}^xh'(t){\Phi}(t)  \, dt  \\
& \quad - \Upsilon'(x) \int_x^{\infty} h'(t) \bar{\Phi}(t)\, dt \\
& \quad + \Upsilon(x) \left(h'(x) {\Phi}(x) + h'(x) \bar{\Phi}(x)
  \right) \\
  & \quad + x e^{x^2/2} \int_x^{\infty} \Upsilon(t) h'(t) e^{-t^2/2}
    \, dt - e^{x^2/2} \Upsilon(x) h'(x) e^{-x^2/2}
\end{align*}
which leads to 
\begin{align}
 g_0'(x) & =   - \Upsilon'(x) \IE[Zh'(Z)]\nonumber \\
  & \quad + \Upsilon'(x)\int_{-\infty}^xh'(t){\Phi}(t)  \, dt -
    \Upsilon'(x)\int_x^{\infty} h'(t) \bar{\Phi}(t)\, dt \nonumber\\
  \label{eq:40}
  & \quad + x e^{x^2/2} \int_x^{\infty} \Upsilon(t) h'(t) e^{-t^2/2}
    \, dt.
\end{align}
Now we can use the fact that 
$ \Upsilon'(x)  = x e^{x^2/2} \int_x^{\infty}e^{-t^2/2}\, dt \le 1
  \mbox{ for all }x \ge 0$
as
well as all the arguments outlined at the previous step to
deduce the bound: 
$ | g_0'(x) |  \le c \left(  \sqrt{\frac{2}{\pi}} +2(x+1/\sqrt{2\pi}) +
               \frac{1}{\sqrt{2\pi}}\right)\le 2c (x+ 1)$
whence 
\begin{align}
  \label{eq:22}
  \frac{ | g_0'(x) |}{x^2+2} \le  c \frac{2x+2}{x^2+2} \le c.
\end{align}
Similar (symmetric) arguments hold also for negative $x$ and thus
$|  g'(x) | \le 4c. $

\

\noindent \underline{Step 4:} a bound on  $\chi(x) = g'(x)/x$. Using \eqref{eq:11}
we know that 
\begin{align}
  \label{eq:2}
  \chi(x) & = \frac{g_0'(x)}{x(x^2+2)} - \frac{1}{(x^2+2)^2} g_0(x).
\end{align}
The second summand in \eqref{eq:2} is bounded using \eqref{eq:31} to
get 
\begin{align}
  \label{eq:12}
  \frac{1}{(x^2+2)^2} \left| g_0(x) \right| \le 3c.
\end{align}
For the first summand we use \eqref{eq:40} to deduce 
\begin{align*}
   \frac{g_0'(x)}{x} & =   - \frac{\Upsilon'(x)}{x} \IE[Zh'(Z)]\nonumber \\
  & \quad + \frac{\Upsilon'(x)}{x}\int_{-\infty}^xh'(t){\Phi}(t)  \, dt -
    \frac{\Upsilon'(x)}{x}\int_x^{\infty} h'(t) \bar{\Phi}(t)\, dt \nonumber\\
  & \quad +  e^{x^2/2} \int_x^{\infty} \Upsilon(t) h'(t) e^{-t^2/2}
    \, dt.
\end{align*}
At this stage it is useful to remark that, for $x\ge0$, the function $
{\Upsilon'(x)}/{x}$ is strictly decreasing with maximal value
$\sqrt{\pi/2}$ and hence 
$  \left|  \frac{g_0'(x)}{x} \right|  \le  c \left( 1 +  2\left( x +
                                      \frac{1}{\sqrt{2\pi}} \right) +
                                      \frac{1}{\sqrt{2\pi}}\right)  \le c \left( 2x+ 3 \right)$
and thus
$   \left|  \frac{g_0'(x)}{x(x^2+2)} \right| \le 3c$
which, combined with \eqref{eq:12}, leads (after applying the
symmetric arguments for $x \le 0$) to
$  \left| \chi(x) \right| \le 6c.$

\

\noindent \underline{Step 5:} a bound on  $\|\chi'\|$. Direct
computations using \eqref{eq:3} 
\begin{align*}
  \chi(x) = \frac{1}{x^2+2} \left( 1-\frac{1}{x^2+2} \right) g_0(x) -
  \frac{x}{x^2+2}\left( h(x) -\IE [Z^2h(Z)]  \right)  
\end{align*}
and thus  
\begin{align*}
  \chi'(x)  & = -\frac{2x^3}{(x^2+2)^2}
  \frac{g_0(x)}{x^2+2} + \left( 1 -\frac{2}{x^2+2}
  \right)\frac{g_0'(x)}{x^2+2}   \\
& \qquad- \frac{2-x^2}{x^2+2} \frac{ h(x) -\IE [Z^2h(Z)]}{x^2+2}  -
  \frac{x}{x^2+2} h'(x). 
\end{align*}
Using the bounds $ \left|{2x^3}/{(x^2+2)^2}\right|\le 1,$
$ \left| 1 -{2}/({x^2+2}) \right| \le 1,$
$ \left|({2-x^2})/({x^2+2})\right| \le 1$ and
$ \left|{x}/({x^2+2})\right| \le 1$ as well as \eqref{eq:10},
\eqref{eq:22} and \eqref{eq:32} we conclude (after applying the
symmetric arguments for $x \le 0$) $ \left| \chi'(x) \right| \le 7c.$

 \end{proof}

\subsection{Verifying bounds on expectations}

In this section we find bounds on the expectations in Theorem \ref{maintheorem} in order to prove Corollary \ref{maincorollary}.   We will make use of the following lemma.

\begin{lemma}\label{lem2} If $x_1,\ldots, x_N$ is the unique strictly decreasing zero-mean solution  of  (\ref{rec-max0}), then $x_1=O(\sqrt{\log N})$.

\end{lemma}
\begin{proof} To simplify the notation, note that it suffices to consider the rescaled recursion $x_{n+1}^3 = x_n^3 -S_n^{-1}$, where $S_n$ is defined in the proof of Lemma \ref{lem1}.  By expressing $x_1^3$ as a  telescoping sum,
\begin{equation*}
 x_1^3= \sum_{n=1}^{m-1}(x_n^3-x_{n+1}^3) +x_m^3= \sum_{n=1}^{m-1}S_n^{-1} +x_m^3
 \le
\sum_{n=1}^{m-1}(n/x_1)^{-1} +x_m^3 \le  x_1(1+\log m) +x_m^3,
\end{equation*}
where we have used Euler's approximation to the harmonic sum for the last inequality.  By the variance property (P2) (in this  rescaled  case $x_1^2+\ldots +x_N^2=N-1$) we have that  $x_1$ is  bounded away from zero (as a sequence indexed by $N$) and $x_m$ is  bounded, so $x_m^2/x_1$ is bounded.   Dividing the above display by $x_1$, we then obtain $x_1=O(\sqrt{\log N})$.\end{proof}

\begin{proof}[{Proof of Corollary \ref{maincorollary}}]
 From Proposition \ref{prop:gener-zero-bias} and the recursion (\ref{rec-max0}), note that $p^{\star}(x)$ puts  mass $1/(N-1)$ on each interval between successive $x_n$, so it is easy to create a coupling of $W \sim \mathbb{P}_N$ with $W^{\star}\sim p^{\star}(x)$ 
such that  $$|W- W^{\star}|\le  |x_{n}-x_{n+1}|$$ when $W\in [x_{n+1},x_n]$. For a detailed proof of such a coupling, see the construction given in \cite{McKeague2015}.  From Lemma \ref{lem2} we then have 
\begin{equation}
\label{b1} \IE |W-W^{\star}|\le {1\over N-1} \sum_{n=1}^{N-1}(x_{n}-x_{n+1})=  {2x_{1}\over N-1}= O\left(\frac{\sqrt{\log N}}{N}\right).
\end{equation}
Second, using $|W|\le x_1=O(\sqrt{\log N})$ it follows immediately that
\begin{equation}
\nonumber \IE [|W||W-W^{\star}|] = O\left(\frac{{\log N}}{N}\right).
\end{equation}

Third,  the zero-median property gives $$2x_m^3=x_m^3-x_{m+1}^3 = {S_m^{-1}}\ge (m/x_m)^{-1} ={x_m}/{m},$$
where  $m=N/2+1$, so $ x_m\ge  1/\sqrt {N}$.
By  symmetry
$$ \IE\left| \frac{1}{W}-\frac{1}{W^{\star}}\right|= \IE \left|\frac{1}{W}-\frac{1}{W^{\star}}\right| 1_{W^{\star}\in (x_{m+1}, x_m]} + 2\sum_{n=1}^{m-1}  \IE \left|\frac{1}{W}-\frac{1}{W^{\star}}\right| 1_{W^{\star}\in (x_{n+1}, x_n]}. $$
  From Proposition \ref{prop:gener-zero-bias} note that $p^{\star}(x) \propto x^2$  for $x\in (x_{m+1}, x_m]$.  Also using the fact  that $p^{\star}(x)$ puts  mass $1/(N-1)$ on this  interval,  the first term above can be written
$$\frac{6}{x_m^3(N-1)}\int_0^{x_m} \left({1\over x}-{1\over x_m}\right)x^2\, dx\le \frac{3}{x_m(N-1)}=O\left({\frac{1}{\sqrt N}}\right).$$
The second term is bounded above by the telescoping sum
$${2\over N-1}\sum_{n=1}^{m-1}\left({1\over x_{n+1}}-{1\over x_n}\right)={2\over N-1} \left({1\over x_{m}}-{1\over x_1}\right)=O\left({\frac{1}{\sqrt N}}\right),$$
so we have
$$ \IE\left| \frac{1}{W}-\frac{1}{W^{\star}}\right|=O\left({\frac{1}{\sqrt N}}\right).$$

Fourth,
\begin{equation}
\nonumber \IE \left|1-\frac{W^{\star}}{W}\right| \le \sqrt {N} \, \IE |W-W^{\star}| = O\left(\sqrt{\frac{\log N}{N}}\right)
\end{equation}
 using $ |W|\ge x_m\ge  1/\sqrt {N}$ and   (\ref{b1}). The Corollary now follows  from Theorem \ref{maintheorem}.
\end{proof}

\section*{Acknowledgements}
The research of Ian McKeague was partially supported by NSF Grant
DMS-1307838 and NIH Grant 2R01GM095722-05. {The research of Yvik Swan
was partially by the Fonds de la Recherche Scientifique - FNRS under
Grant no F.4539.16} as well as IAP Research Network P7/06 of the
Belgian State (Belgian Science Policy).  We also thank the Institute
for Mathematical Sciences at National University of Singapore for
support during the {\it Workshop on New Directions in Stein's Method}
(May 18--29, 2015) where work on the paper was initiated.

\end{document}